\documentclass[11pt]{amsart}
\usepackage{amsmath}
\usepackage{amsfonts}
\usepackage{amssymb}
\usepackage[normalem]{ulem} 

\usepackage[utf8]{inputenc}
\usepackage[T2A]{fontenc}
\usepackage[russian,english]{babel}
\inputencoding{utf8}

\newtheorem{theorem}{Theorem}[section]

\newtheorem{corollary}[theorem]{Corollary}
\newtheorem{proposition}[theorem]{Proposition}
\theoremstyle{definition}
\newtheorem{definition}[theorem]{Definition}

\theoremstyle{remark}

\numberwithin{equation}{section}
\numberwithin{equation}{section}

\title[Meromorphic modular differential equations]
{Meromorphic modular differential equations for the elliptic genus of Calabi--Yau varieties}
\author{Dmitrii Adler}
\address{Department of Mathematics and Computer Science, Saint Petersburg State University, Saint Petersburg}
\email{dmitry.v.adler@gmail.com}
\author{Valery Gritsenko}
\address{International laboratory of mirror symmetry and automorphic forms, HSE University, Moscow and Laboratoire Paul Painlev\'e, Universit\'e de Lille}
\email{valery.gritsenko@univ-lille.fr}
\date{\today}
\begin{document}
\maketitle

\begin{abstract} 
We study differential equations -- not necessarily monic -- for Jacobi modular forms.  As an application, we derive a second-order differential equation for the elliptic genus of a $K3$ surface involving the logarithmic derivative of the modular invariant $j(\tau)$, and we explicitly find a third-order equation
with a leading coefficient of weight 12  for the elliptic genus of any four-dimensional Calabi--Yau variety.  Furthermore, we propose a simple recursive algorithm for investigating second-order differential equations with leading coefficient $E_4(\tau)$ or $E_6(\tau)$.
\end{abstract}

\section{Introduction: MDEs and MMDEs}
Let $\varphi_{k,m}(\tau, z)$ be a weak Jacobi form of weight $k$ and index $m$, where $\tau \in \Bbb {H}$ lies in the upper half-plane and $z\in \Bbb C$. 
A  {\it modular differential equation} (or {\it MDE}) {\it of order $n$} is a differential equation of the form
\begin{equation}\label{MDEdef}
H_k^{[n]}(\varphi_{k,m})+\sum_{t=1}^{n-2} f_{2(n-t)}(\tau)H_k^{[t]}(\varphi_{k,m})=0
\end{equation}
where $H_k^{[n]}$ denotes the $n$th iterate, defined in \eqref{Hn-iter}, of the modular differential operator $H_k$ for Jacobi forms. This operator generalizes the heat operator (see \eqref{Hk}), and each coefficient $f_a\in M_a(SL_2(\Bbb Z))$ is a modular form of weight $a$.

We studied MDEs in a series of papers \cite{AG22}--\cite{AG25}. The choice of weak Jacobi forms is motivated by the polynomial nature of their bigraded ring. It is known (see \cite{EZ}) that a Jacobi  form can be represented as a vector-valued modular form.  
The general theory of modular differential equations for vector-valued modular forms one can find in \cite{M},  \cite{FM} and \cite{Y}.
Our approach differs in that we replace differentiation with respect to $\tau$ by the heat operator. To prove that a Jacobi form satisfies a modular differential equation, it suffices to carry out elementary polynomial calculations involving its Fourier coefficients at fixed powers of $q=e^{2\pi i\tau}$, viewed as Laurent polynomials in $\zeta=e^{2\pi iz}$. This makes it straightforward to investigate possible modular differential equations for the elliptic genera of complex varieties with vanishing first Chern class, provided their dimension is not too large (for example, up to complex dimension 13). In particular, we derived a third-order MDE satisfied by the elliptic genera of K3 surfaces and Calabi--Yau fourfolds with Euler characteristic 48 or $-18$ (see \cite{AG23}). 
For $N=2$ superconformal field theories, elliptic genera are a natural generalization of chiral characters. Moreover, elliptic genera can be computed geometrically in terms of Gromov--Witten invariants in dual string compactifications (see \cite{OP}).

As in the study of differential equations for characters of vertex operator algebras \cite{KNS}, we focused on monic differential equations, that is, equations with leading coefficient 1. The equations we obtained in \cite{AG22}--\cite{AG25}  for Jacobi forms have the lowest possible order among monic equations.
In this paper, we investigate differential equations that are not necessarily monic, namely, equations of the form
\begin{equation}\label{MDEdef}
f_s(\tau)H_k^{[n]}(\varphi_{k,m})+\sum_{t=1}^{n-2} f_{2(n-t)+s}(\tau)H_k^{[t]}(\varphi_{k,m})=0.
\end{equation}
We call such an equation a {\it meromorphic modular differential equation} ({\it MMDE} or 
{\it $f_s$-MMDE}) of order $n$ with leading coefficient $f_s(\tau)$, since dividing by $f_s(\tau)$ yields a differential equation whose coefficients are meromorphic modular forms. In particular, we derive a second-order differential equation for the elliptic genus of a K3 surface involving the logarithmic derivative of the modular invariant $j(\tau)$ (see Proposition \ref{Eq-phi01} and \eqref{MEphi01}).

We propose a simple recursive algorithm for investigating second-order differential equations with leading coefficient $E_4(\tau)$ or $E_6(\tau)$. As a corollary, we show that the second-order MMDE for the elliptic genus of K3 surfaces is highly exceptional  (see Theorem \ref{E-CY2d}). In Theorem 5.1, we also obtain an explicit third-order MMDE for the elliptic genus of any Calabi--Yau fourfold or, more generally, for any weak Jacobi form of weight $0$ and index $2$. Up to multiplication by a nonzero constant, there are exactly three special forms satisfying the monic third-order MDEs found in \cite{AG23}, one additional form each satisfying an $E_4$-MMDE and an $E_8$-MMDE, and two additional forms satisfying $E_6$-MMDEs. For a generic weak Jacobi form of weight $0$ and index $2$, the leading coefficient $f_s$ of the third-order equation \eqref{MDEdef} has weight $12$.

\section{Examples of MMDEs}

\subsection{The graded ring of weak Jacobi forms and the operator $H_k$}

The elliptic genus of a Calabi--Yau variety of complex dimension $d$ is a weak Jacobi form of weight $0$ and index $\frac{d}{2}$ with integral Fourier coefficients.
For completeness, we recall the necessary definitions.

\begin{definition}\label{def:JF}
Let $k\in \Bbb Z$ and $m\in \frac{1}{2}\Bbb N$.
A holomorphic function $\varphi:\Bbb H\times\Bbb C\to\Bbb C$ is a weak Jacobi form of weight $k$ and index $m$ if it satisfies the transformation laws
\begin{align*}
\varphi\bigl(\frac{a\tau+b}{c\tau+d},\frac{z}{c\tau+d}\bigr)&
=(c\tau+d)^ke^{2\pi i m \frac{cz^2}{c\tau+d}}\,\varphi(\tau,z),\\
\varphi(\tau,z+x\tau+y)&=(-1)^{2m(x+y)}\,e^{-2\pi i m (x^2\tau+2xz)}\varphi(\tau,z)
\end{align*}
for all $\left(\begin{smallmatrix}
a & b \\
c & d \end{smallmatrix}\right)\in SL_2(\Bbb Z)$ and $x,y\in\Bbb Z$,
and admits a Fourier expansion of the form
\begin{equation*}
\varphi(\tau,z)=\sum_{\substack{n\geq 0,\ l\in \frac{1}{2}\Bbb Z}} a(n,l)\exp(2\pi i(n\tau+lz)),
\end{equation*}
where the exponents $l$ are integral when $m$ is integral and half-integral when $m$ is half-integral.
\end{definition}

We denote by $J_{k,m}^w$ the finite-dimensional vector space of weak Jacobi forms of weight $k$ and index $m$. For $m=0$, the space $J_{k,0}^w=M_k(SL_2(\Bbb Z))$ coincides with the space of modular forms of weight $k$.

Jacobi forms of half-integral index are related to the odd Jacobi theta series $\vartheta(\tau,z)$, which is a Jacobi form of weight $\frac{1}{2}$ and index $\frac{1}{2}$ with a multiplier system of order $8$ (see \cite{GN98}):
\begin{equation*}\label{theta}
\begin{split}
\vartheta(\tau ,z)&=
q^{1/8}(\zeta^{1/2}-\zeta^{-1/2})\prod_{n\ge 1}\,(1-q^{n}\zeta)(1-q^n 
\zeta^{-1})(1-q^n)\\
{}&=q^{\frac{1}{8}}\zeta^{\frac{1}{2}}
\sum_{n\in \Bbb Z}\,(-1)^nq^{\frac{n(n+1)}{2}}\zeta^{n}, \quad
q=e^{2\pi i \tau},\ \zeta=e^{2\pi i z}.
\end{split}
\end{equation*}
The theta series is related to two other important functions through the identities
\begin{equation*}\label{G2}
\frac{\partial\vartheta(\tau ,z)}{\partial z}\big|_{z=0}
=2\pi i \,\eta(\tau )^3,\quad
\dfrac{\partial^2 \log \vartheta(\tau,z)}{\partial z^2}=
-\wp(\tau,z)-\frac{\pi^2}3 E_2(\tau),
\end{equation*}
where  
$\eta(\tau)\in M_{\frac{1}2}(SL_2(\Bbb Z),v_\eta)$ is the Dedekind eta function ($v_\eta^{24}=1$), $\wp(\tau,z)$ is the Weierstrass $\wp$-function, and 
$ 
E_2(\tau)=1-24\sum_{n\ge 1 }\sigma_1(n)q^n
$
is the quasi-modular Eisenstein series of weight $2$.
Using 
$\vartheta(\tau ,z),$ we obtain two main generators of the graded ring  weak Jacobi forms
\begin{equation*}\label{phi-21}
\varphi_{-2,1}=\frac{\vartheta(\tau,z)^2}{\eta(\tau)^6}
=\zeta^{\pm 1}-2+ 
q(-2\zeta^{\pm 2} +8\zeta^{\pm 1}-12)+O(q^2),
\end{equation*}
\begin{equation}\label{phi01}
\varphi_{0,1}=-\frac{3}{\pi^2}\wp(\tau,z)\varphi_{-2,1}=
\zeta^{\pm 1}+10+ q(10\zeta^{\pm 2} - 64\zeta^{\pm 1}+108)+O(q^2)
\end{equation}
where $\zeta^{\pm l}= \zeta^{l}+\zeta^{-l}$. Three other generators are defined below.
 
By \cite[Theorem 9.3]{EZ}, the bigraded ring of weak Jacobi forms of even weight and integral index is a polynomial ring in four generators:
\begin{equation}\label{J2*}
J_{2*,*}^w=\bigoplus_{k\in \Bbb Z,\, m\in \Bbb Z_\ge 0}  
J_{2k,m}^w=
\Bbb C[E_4, E_6, \varphi_{-2,1},\varphi_{0,1}]
\end{equation}
where 
$E_4(\tau)=1+240\sum_{n\ge 1}\sigma_3(n)q^n$ and 
$E_6(\tau)=1-504\sum_{n\ge 1}\sigma_5(n)q^n$
are the  Eisenstein series of weight $4$ and $6$, respectively.
 
For weak Jacobi forms of odd weight or half-integral index, the following relations hold (see \cite{EZ} and \cite{GN98}):
$$
J_{2k+1,m}^w=\varphi_{-1,2}\cdot J_{2k+2,m-2}^w,\quad
J_{2k+1,m+\frac{1}{2}}^w=\varphi_{-1,\frac{1}{2}}\cdot J_{2k+2,m}^w,
$$
$$
J_{2k,m+\frac{1}{2}}^w=\varphi_{0,\frac{3}{2}}\cdot J_{2k,m-1}^w
$$ 
for all
$k\in \Bbb Z$, $m\in \Bbb N$,
where 
\begin{equation*}\label{phi32}
 \varphi_{-1,\frac{1}{2}}(\tau,z)=\frac{\vartheta(\tau,z)}
{\eta^3(\tau)}, \ \varphi_{-1,2}(\tau,z)=
\frac{\vartheta(\tau,2z)}{\eta^3(\tau)},\ 
\varphi_{0,\frac{3}{2}}(\tau,z)=\frac{\vartheta(\tau,2z)}
{\vartheta(\tau,z)}.
\end{equation*}

The modular differential operator  $D_k: M_k\to M_{k+2}$ is defined by
$$
D_k(f(\tau))=D(f(\tau))-\frac{k}{12}E_2(\tau)\cdot f(\tau), \qquad D=q\frac{d\ }{dq}=\frac{1}{2\pi i}\frac{d\ }{d\tau}.
$$
An analogous differential operator on weak Jacobi forms is defined by
\begin{equation}\label{Hk}
H_k:  J_{k,m}^w\to  J_{k+2,m}^w, \quad H_{k}(\varphi_{k,m})=
H^{(m)}(\varphi_{k,m}) -\frac{2k-1}{2}E_2(\tau)\varphi_{k,m},
\end{equation}
where $H^{(m)}$ is the renormalized {\it heat operator}  
\begin{equation}\label{heat}
{H}^{(m)}=\frac{3}{m}\frac 1{(2\pi i)^2}
\left( 8\pi i m\dfrac{\partial\ }{\partial \tau}-
\dfrac{\partial^2 \  }{\partial z^2}\right)
=12q\frac {d\ }{dq}- \frac{3}m\left(\zeta\frac {d\ }{d\zeta}\right)^2.
\end{equation} 
We choose this normalization to simplify calculations. Throughout the paper, we write $H=H^{(m)}$, since the index $m$ will be clear from the context. The operator $H_k$ maps weak, holomorphic, and cusp Jacobi forms to Jacobi forms of the same type. 
In terms of Fourier coefficients, its action is given by (see \cite[Proposition 2.1]{AG23})
\begin{multline}\label{Hk-act}
 H_k(\varphi_{k,m})(\tau, z)=
 \sum_{n\ge 0,\,l} q^n\Bigl[
  \sum_l \Bigl(\frac 3m(4nm-l^2)-\frac{2k-1}2\Bigr)a(n,l)\zeta^l\\
  \qquad +12(2k-1)\sum_{s=1}^n\sigma_1(s)a(n-s,l)\zeta^l \Bigr].
\end{multline}
Iterating this first-order operator with respect to the heat operator yields the higher-order operators
\begin{align}\label{Hn-iter}
H^{[n]}_k=H_{k+2(n-1)}\circ \ldots \circ H_k: J_{k,m}^w\to J_{k+2n,m}^w.
\end{align}

\subsection{MMDEs for the elliptic genera of K3 surfaces and Calabi--Yau fivefolds}

In \cite{AG22}--\cite{AG25}, we showed that the elliptic genus of a Calabi--Yau variety does not satisfy an MDE of order one or two unless the variety has complex dimension $3$. We also found third-order MDEs for the elliptic genera of K3 surfaces, arbitrary Calabi--Yau fivefolds, and two Calabi--Yau fourfolds with Euler characteristics $48$ and $-18$ (see \cite[\S 5.3]{AG23}). In this subsection, we show that in some cases the order of the differential equation can be reduced by passing to MMDEs.

We recall the formulae for the elliptic genera of an arbitrary K3 surface and a Calabi--Yau fivefold (see \cite{Gr99}):
$$
EG(K3;\tau, z)=2\varphi_{0,1}(\tau,z),\quad
EG(CY_5;\tau, z)=\frac{e(CY_5)}2\varphi_{0,\frac 52}(\tau,z)
$$
where $e(CY_5)$ denotes the Euler characteristic of $CY_5$, $\varphi_{0,1}$ is defined in \eqref{phi01}, and  $\varphi_{0,\frac 52}=\varphi_{0,1}\varphi_{0,\frac 32}$.
In \cite{AG22}, we obtained the following monic MDEs of minimal order:
\begin{equation*}\label{eqK3}
\bigl(H_0^{[3]}-\tfrac{101}{4}E_4H_0+10E_6\bigr)\varphi_{0,1}=0,
\ \ 
\bigl(H_0^{[3]}-\tfrac{611}{25}E_4H_0+\tfrac{88}{25}E_6\bigr)\varphi_{0,\frac 52}=0.
\end{equation*}

\begin{proposition}\label{Eq-phi01} 
The following identities hold: 
\begin{equation*}\label{eq2}
\bigl(E_4H_0^{[2]}+4E_6H_0-\tfrac{5}{4}E_4^2\bigr)\varphi_{0,1}=0,\quad
\bigl(E_4H_0^{[2]}+4E_6H_0-\tfrac{11}{25}E_4^2\bigr)\varphi_{0,\frac{5}{2}}=0.
\end{equation*}
\end{proposition}
\begin{proof} We include a proof to recall the basic formulae from \cite{AG22}. By \eqref{Hk-act}, the action of the modular differential operator 
$H_k$ on $q^0$-term  of the Fourier expansion of a Jacobi form, $q^0(\varphi_{k,m})=\sum_{l}\,a(0,l)\zeta^l$, is particularly simple:
\begin{equation}\label{Hkq0}
q^0(H_k(\varphi_{k,m}))= 
\sum_{l}\,-\bigl(\frac{3l^2}m+\frac{2k-1}2\bigr)a(0,l)\zeta^l.
\end{equation}
Since the spaces $J_{-2,1}^w$, $J_{0,1}^w$, and $J_{2,1}^w$ are one-dimensional, comparison of the $q^0$-terms gives
\begin{equation*}\label{Hphi}
H_{-2}(\varphi_{-2,1})=-\frac 12\varphi_{0,1},\quad
H_{0}(\varphi_{0,1})=-\frac 52 E_4\varphi_{-2,1}.
\end{equation*}
We also have the product rule (see \cite[Proposition 2.2]{AG22})
\begin{equation}\label{H{k1+k2}}
H_{k_1+k_2}(f_{k_1}\varphi_{k_2,m})=
12D_{k_1}(f_{k_1})\varphi_{k_2,m}
+f_{k_1}H_{k_2}(\varphi_{k_2,m})
\end{equation}
for $f_{k_1}(\tau)\in M_{k_1}(SL_2(\Bbb Z))$ and
$\varphi_{k_2,m}(\tau,z)\in J_{k_2,m}$. 
Setting $\Bbb D_k=12D_k$, we obtain
\begin{equation}\label{DkE}
\Bbb D_4(E_4)=-4E_6,\quad
\Bbb D_6(E_6)=-6E_4^2,\quad
\Bbb D_8(E_4^2)=-8E_4E_6.
\end{equation}
Consequently,
\begin{equation*}\label{deq-2:phi1}
H_0^{[2]}(\varphi_{0,1})=-\frac{5}2H_2(E_4\varphi_{-2,1})=
10E_6\varphi_{-2,1}+\frac 54 E_4\varphi_{0,1}.
\end{equation*}
Similarly, the differential equation for $\varphi_{0,\frac{5}{2}}$ follows from the identities
\begin{equation*}\label{H0phi52}
H_0(\varphi_{0,\frac 52})=-\frac{11}5 E_4 \varphi_{-2,1}\varphi_{0,\frac 32},\quad
H_{-2}(\varphi_{-2,1}\varphi_{0,\frac 32})=-\frac{1}5\varphi_{0,1}\varphi_{0,\frac 32}.
\end{equation*}
 \end{proof}
In \cite[Corollary 5.4]{AG25}, we derived a recurrence relation for the Fourier coefficients of $\varphi_{0,1}$ using a third-order MDE. The second-order MMDE in Proposition \ref{Eq-phi01} yields a new recurrence relation.
\begin{corollary} Let 
$$
\varphi_{0,1}(\tau,z)=
\sum_{\substack{ n\geq 0,\,l\in \Bbb Z}}a(n,l)q^n\zeta^l=
\zeta^{\pm 1}+10+ q(10\zeta^{\pm 2} - 64\zeta^{\pm 1}+108)+O(q^2).
$$
Then the Fourier coefficients $a(n, l)$ satisfies the following recurrence relation:
\begin{multline*} 
(4n-l^2)((4n-l^2)+1)a(n,l)=\\
4\sum_{s=1}^n\Biggl(-60(4(n-s)-l^2)^2\sigma_3(s)+2(4(n-s)-l^2)C_1(s)+C_2(s)\Biggr)a(n-s,l)
\end{multline*}
where
$$C_1(s)=-\sigma_1(s)+10\sigma_3(s)+84\sigma_5(s)-240\sum_{t=1}^{s-1}\sigma_1(t)\sigma_3(s-t),$$
\begin{multline*}
C_2(s)=\sigma_1(s)+25\sigma_3(s)+28\sigma_5(s)+4\sum_{t=1}^{s-1}\sigma_1(t)\sigma_1(s-t)
-80\sum_{t=1}^{s-1}\sigma_1(t)\sigma_3(s-t)\\
-672\sum_{t=1}^{s-1}\sigma_1(t)\sigma_5(s-t)
+2800\sum_{t=1}^{s-1}\sigma_3(t)\sigma_3(s-t)
+960\sum_{\substack{r,t,u\ge1\\r+t+u=s}}
\sigma_1(r)\sigma_1(t)\sigma_3(u).
\end{multline*}
\end{corollary}
 
We can rewrite the equations in the proposition using the meromorphic modular form $E_6/E_4$ of weight $2$. In particular,
\begin{equation}\label{MEphi01}
\bigl(H_0^{[2]}+4\frac{E_6}{E_4}H_0-\frac{5}{4}E_4\bigr)\varphi_{0,1}=0
\end{equation}
where the quotient $E_6/E_4$ can be expressed in terms of the modular invariant
$j(\tau)={E_4^3(\tau)}/{\Delta(\tau)}=q^{-1}+744+196884q+\dots$ as
\begin{equation*}\label{E6/E4}
\frac {E_6(\tau)}{E_4(\tau)}=-\frac 1{2\pi i}\frac{j'(\tau)}{j(\tau)}=
\sum_{n\ge 0} b(n)q^n=
1-744q+159768q^2+\dots .
\end{equation*}
Equation \eqref{MEphi01} yields a recurrence relation for the Fourier coefficients of $\varphi_{0,1}$. However, this relation is more complicated than the corresponding recurrence for solutions of the Kaneko--Zagier equation  
$
H_k^{[2]}(\varphi_{k,m})=\lambda E_4 \varphi_{k,m}
$
given in \cite[Proposition 3.2]{AG23}.

The MMDE \eqref{MEphi01} can be compared with the modular differential equations for two-character CFTs with $\ell=2$ in \cite[(4.4)]{HM}, where the parameter $\ell$ characterizes the number of zeros of the Wronskian of the solutions.

\subsection{Further examples of MMDEs of order 2}

Certain reflective Jacobi forms of weight $0$ serve as input for the construction of Lorentzian Kac--Moody algebras of Borcherds type (see \cite{GN98}).
The nearly holomorphic Jacobi form 
$$
\psi_{0,1}^{(-1)}=q^{-1}+(\zeta^{\pm 2}+70)+q(70\zeta^{\pm 2}+32384\zeta^{\pm 1}+131976)+O(q^2)
$$
determines the simplest Lorentzian Kac--Moody algebra of rank 3.
The corresponding Borcherds product is the Igusa modular form of weight $35$. In \cite{AG22}, we proved that 
\begin{equation*}\label{eq5}
\bigl(H_0^{[3]}-\frac{533}{4}E_4H_0+874E_6\bigr)\psi_{0,1}^{(-1)}=0.
\end{equation*}
A straightforward calculation gives the second-order MMDE
\begin{equation*}\label{eq-psi}
\bigl(E_4H_0^{[2]}+4E_6H_0-\frac{437}{4}E_4^2\bigr)\psi_{0,1}^{(-1)}=0.
\end{equation*}
The Eisenstein--Jacobi series $E_{4,1}(\tau,z)$ of weight $4$ and index $1$ satisfies the second-order MDE
$$
\bigl(H_4^{[2]}-\frac{77}4 E_4\bigr)E_{4,1}=0,
$$
which is an equation of Kaneko--Zagier type. By contrast, the minimal order of a monic MDE satisfied by $E_{6,1}(\tau,z)$ is 3
(see \cite{AG22})
\begin{equation*}\label{eq7}
\bigl(H_6^{[3]}-\frac{173}{4}E_4H_6+154E_6\bigr)E_{6,1}=0.
\end{equation*}
The order can be reduced by passing to an MMDE:
\begin{equation*}\label{eq-E61}
\bigl(E_4H_6^{[2]}+4E_6H_6-\frac{77}{4}E_4^2\bigr)E_{6,1}=0.
\end{equation*}
The three additional generators of the graded ring of Jacobi forms, defined after \eqref{J2*}, lie in the kernels of the corresponding modular differential operators. We conclude this section with second-order MMDEs for two related functions:
$$\varphi_{-1,\frac{3}{2}}=\varphi_{-1,\frac{1}{2}}\varphi_{0,1}=\zeta^{\pm \frac{3}{2}}+8\zeta^{\pm \frac{1}{2}}+q(9\zeta^{\pm \frac{5}{2}}-81\zeta^{\pm \frac{3}{2}}+198\zeta^{\pm \frac{1}{2}})+O(q^2),$$
and
\begin{multline*}\varphi_{-1,3}=\varphi_{-1,\frac{3}{2}}\varphi_{0,\frac{3}{2}}=\varphi_{-1,2}\varphi_{0,1}=\varphi_{-1,\frac{1}{2}}\varphi_{0,\frac{3}{2}}\varphi_{0,1}=\\
\zeta^{\pm 2}+10\zeta^{\pm 1}+q(-\zeta^{\pm 4}-62\zeta^{\pm 2}+128\zeta^{\pm 1})+O(q^2).
\end{multline*}
These forms satisfy the third-order MDEs
\begin{equation*}\label{eq32}
(H_{-1}^{[3]}-27E_4H_{-1}+24E_6)\varphi_{-1,\frac{3}{2}}=0,\quad
\bigl(H_{-1}^{[3]}-\frac{101}{4}E_4H_{-1}+10E_6\bigr)\varphi_{-1,3}=0.
\end{equation*}
Both forms also satisfy second-order MMDEs:
\begin{equation*}\label{eq-phi13/2}
\bigl(E_4H_{-1}^{[2]}+4E_6H_{-1}-3E_4^2\bigr)\varphi_{-1,\frac{3}{2}}=0,
\quad \bigl(E_4H_{-1}^{[2]}+4E_6H_{-1}-\frac{5}{4}E_4^2\bigr)\varphi_{-1,3}=0.
\end{equation*}
Note that the coefficients of the differential equations for $\varphi_{-1,3}$ coincide with those of the corresponding equations for $\varphi_{0,1}$.

In Theorem 4.2, we explain why all six second-order MMDEs in this section have the same coefficient $4$ multiplying the first-order term $E_6H_k$.

\section{First-order MMDEs}

In this section, we consider the simplest first-order MMDEs with leading coefficient $E_4$ or $E_6$.

\begin{theorem}\label{Ind1}
Let $\varphi_{k,m}$ be a nonzero weak Jacobi form of weight $k\in \Bbb Z$ and index $m\in \Bbb Z/2$.

1) If $\varphi_{k,m}$ satisfies 
$\bigl(E_4H_{k}+\lambda E_6\bigr)\varphi_{k,m}=0$,
then $\lambda=4b$ for some $b\in\Bbb Z_{\ge 0}$ and, up to multiplication by a constant,
$$
\varphi_{k,m}=\begin{cases}
\varphi_{0,\frac 32}(\tau, az)E_4^b\,\Delta^c, 
\  m=\frac 32 a^2,
\\
\varphi_{-1,\frac 12}(\tau, az)E_4^b\,\Delta^c,\ m=\frac 12 a^2.
\end{cases}
$$

2) If $\varphi_{k,m}$ satisfies
$\bigl(E_6H_{k}+\lambda E_4^2\bigr)\varphi_{k,m}=0$,
then $\lambda=6b$ for some $b\in\Bbb Z_{\ge 0}$ and, up to multiplication by a constant,
$$
\varphi_{k,m}=\begin{cases}
\varphi_{0,\frac{3}{2}}(\tau,az)E_6^b\Delta^c,
& m=\frac{3}{2}a^2,\\
\varphi_{-1,\frac{1}{2}}(\tau,az)E_6^b\Delta^c,
& m=\frac{1}{2}a^2,
\end{cases}
$$
where $a\in\Bbb Z_{>0}$ and $c\in\Bbb Z_{\ge 0}$.
\end{theorem}
\begin{proof}
In \cite[Theorem 2.2]{AG23} we proved that the kernel of the modular differential operator $H_k$ is generated by 
the holomorphic Jacobi forms of weight  $1/2$. 
More precisely, 
\begin{align}\label{KerH}
\hbox{Ker\,} (H_k)
=\begin{cases}
\Bbb C\varphi_{0,\frac 32}(\tau, az)\Delta(\tau)^c,\quad
\ k=12c\ge 0,\  a\in \Bbb N,\\
\Bbb C\varphi_{-1,\frac 12}(\tau, az)\Delta(\tau)^c,\quad
k=12c-1,
\end{cases}
\end{align}
where $a\in\Bbb Z_{>0}$, $c\in\Bbb Z_{\ge 0}$, and $\Delta(\tau)\in S_{12}(SL_2(\Bbb Z))$ is the Ramanujan delta function. Recall that $D_{12}(\Delta)=0$ and that $H_k$ is $\Delta$-homogeneous (see \eqref{H{k1+k2}}).

If $\lambda=0$, the assertion of the theorem is equivalent to \eqref{KerH}. We therefore assume that $\lambda\ne 0$.

Recall that the divisors of $E_4(\tau)$  and  $E_6(\tau)$ are
$\tau=e^{2\pi i/3}$ and  $\tau=i$, with multiplicities $1/3$ and $1/2$, respectively. Then  
\begin{align}\label{multE4}
\left(E_4H_{k}+\lambda E_6\right)\varphi_{k,m}=0\quad {\rm implies}
\quad 
\varphi_{k,m}=E_4\psi_{k-4,m}.
\end{align}
If $k$ is even and $m$ is integral this follows directly from the fact
that $E_4$, $E_6$, $\varphi_{0,1}$ and $\varphi_{-2,1}$ freely generate the bigraded ring of weak Jacobi forms (see  \eqref{J2*}). The squares of three other generators mentioned after
\eqref{J2*}  belong to $J_{2*,*}^w$. Thus, the same conclusion holds when $k$ is odd or $m$ is half-integral.

Substituting $\varphi_{k,m}=E_4\psi_{k-4,m}$  into the equation, applying \eqref{H{k1+k2}}, and cancelling the factor $E_4$, we obtain
$$\bigl(E_4H_{k-4}+(\lambda-4)E_6\bigr)\psi_{k-4,m}=0.$$
If $\lambda=4$, then \eqref{KerH} implies that, up to multiplication by a constant,
$\psi_{k-4,m}=\Delta^c\varphi_{0,\frac 32}(\tau, az)$ or 
$\Delta^c\varphi_{-1,\frac 12}(\tau, az)$.
If $\lambda\neq 4$, we repeat this procedure. After $b$ iterations, we obtain a form $\psi_{k-4b, m}$ satisfying
$$
\bigl(E_4H_{k-4b}+(\lambda-4b)E_6\bigr)\psi_{k-4b,m}=0.
$$
By the structure theorem \eqref{J2*}, the weights of nonzero weak Jacobi forms of fixed index are bounded below. Thus, the process must terminate with $\lambda=4b$; otherwise, for sufficiently large $b$, we would have $\psi_{k-4b,m}=0$, contradicting the assumption that $\varphi_{k,m}\ne 0$.

The proof of the second assertion, concerning
$\bigl(E_6H_k+\lambda E_4^2\bigr)\varphi_{k,m}=0$,
is completely analogous.
\end{proof}

Note that every solution of the MMDEs in Theorem \ref{Ind1} has the split form $g(\tau)\varphi_{0,\frac{3}{2}}(\tau,az)$ or $g(\tau)\varphi_{-1,\frac{1}{2}}(\tau,az)$, where the modular form $g(\tau)$ satisfies a corresponding differential equation with $H_k$ replaced by $D_k$.

\begin{proposition}\label{Hkphi} Consider the equations
$$
H_{k-2}(\varphi_{k-2,m})=f_k(\tau)\varphi_{0,\frac 32}(az),\quad
H_{k-3}(\varphi_{k-3,m})=f_k(\tau)\varphi_{-1,\frac 12}(az),
$$
where $m=\frac{3a^2}{2}$ and $m=\frac{a^2}{2}$, respectively, and
$f_k(\tau)=1+O(q)\in M_k(SL_2(\Bbb Z))$.
Every weak Jacobi form satisfying either equation, if such a solution exists, has split form:
$\varphi_{k-2,m}=g_{k-2}(\tau)\varphi_{0,\frac 32}(az)$ or 
$\varphi_{k-3,m}=g_{k-2}(\tau)\varphi_{-1,\frac 12}(az)$.
\end{proposition}
\begin{proof}
We prove the assertion for the first equation when $k$ is even.
If $\varphi_{k-2,m}$ and $\psi_{k-2,m}$
are two solutions, then 
\begin{align}
\varphi_{k-2,m}-\psi_{k-2,m}\in 
\hbox{Ker\,}(H_{k-2})=\Bbb C\varphi_{0,\frac 32}(az)\,\Delta^c 
\end{align}
where $\varphi_{0,\frac 32}(az)=\zeta^{\frac a2}+\zeta^{-\frac a2}+O(q)$.
Thus, if one solution has split form, so does every other solution.
We prove the proposition by induction on the weight, beginning with several special cases.

Let $k=0$. Then $f_k(\tau)\equiv 1$. By \eqref{Hkq0}, the multiplier of each Fourier coefficient $a(0,l)$ in $q^0[H_{-2}(\varphi_{-2,m})]$ is
$\frac{5}{2}-\frac{2l^2}{a^2}\ne 0$.
Consequently, up to a nonzero constant,
$$q^0[\varphi_{-2,m}]=\zeta^{\frac {a} 2}+ \zeta^{-\frac {a} 2}.$$
This contradicts the fact that $\varphi_{-2,m}$ vanishes to order at least $2$ at $z=0$. Hence the equation
\begin{align}\label{eq-phi32}
H_{-2}(\varphi_{-2,m})=\varphi_{0,\frac{3}{2}}(\tau,az)
\end{align}
has no solutions.

By \eqref{Hkq0}, when $k>2$, every Fourier coefficient in the $q^0$-term of $\varphi_{k-2,m}$ is multiplied by a nonzero factor under $H_{k-2}$. The case $k=4$ is therefore similar to the case $k=0$: the equation
$$
H_2(\varphi_{2,m})=E_4\varphi_{0,\frac{3}{2}}(\tau,az)
$$
has no solutions, since $\hbox{ord}_{z=0}(\varphi_{2,m})\ge 2$
by \eqref{J2*}.

For $k=6$, $8$ and $10$, the equation
$H_{k-2}(\varphi_{2,m})=E_k\varphi_{0,\frac 32}(az)$ has a split solution by \eqref{DkE}.

Let $k\ge 12$. As above, comparing the $q^0$-terms in
$$
H_{k-2}(\varphi_{k-2,m})=f_k\,\varphi_{0,\frac 32}(az), \qquad
f_k(\tau)=1+O(q),
$$
gives
$q^0[\varphi_{k-2,m}]=-(k-2)^{-1}(\zeta^{\frac a2}+\zeta^{-\frac a2})$. 
Choose a modular form $g_{k-2}$ of weight $k-2$ with constant term $-(k-2)^{-1}$. For example, we may take $g_{k-2}=-(k-2)^{-1}E_4^d$ or 
$-(k-2)^{-1}E_4^{d-1}E_6$ if $k-2=4d$ or $4d+2$, respectively.  Then 
$$
\Bbb D_{k-2}(E_4^d)=-4dE_4^{d-1}E_6,\quad
\Bbb D_{k-2}(E_4^{d-1}E_6)=-4(d-1)E_4^{d-2}E_6^2-6E_4^{d+1},
$$
and $f_k-{\Bbb D}_{k-2}(g_{k-2})=\Delta f_{k-12}$.
Similarly,
$$\varphi_{k-2,m}-g_{k-2}\varphi_{0,\frac 32}(az)=\Delta\psi_{k-14,m}$$
for some weak Jacobi form $\psi_{k-14,m}$.
By the $\Delta$-homogeneity of the modular differential operator, we obtain
$$
H_{k-14}(\psi_{k-14,m})=f_{k-12}\,\varphi_{0,\frac 32}(az).
$$
Iterating this process, dividing by powers of $\Delta$ when necessary, either produces a split solution of the original equation or shows that no solution exists. The proof for the second equation is entirely analogous.
\end{proof}

\begin{corollary}\label{KerH2} If
$H_k^{[2]}(\varphi_{k,m})=0$, then 
$\varphi_{k,m}\in \hbox{Ker\,} H_k$.
\end{corollary}
\begin{proof} This follows directly from \eqref{KerH} and \eqref{eq-phi32}.
\end{proof}

\section{Second-order MMDEs}

In this section, we investigate the simplest second-order MMDEs with leading coefficient $E_4$:
\begin{equation}\label{E4ME}
\bigl(E_4H_k^{[2]}+\mu E_6H_k+\lambda E_4^2\bigr)\varphi_{k,m}=0.
\end{equation}
When $\mu=0$, cancelling the factor $E_4$ yields an equation of Kaneko--Zagier type (see \cite{K}, \cite{AG23}):
\begin{equation}\label{KZ}
\bigl(H_k^{[2]}+\lambda E_4\bigr)\varphi_{k,m}=0.
\end{equation}
We studied this equation in \cite[\S 3]{AG23}. The first author will give a complete description of its solutions for even $k$ in forthcoming work. In this section, we show that every non-split solution of \eqref{E4ME} is generated by a non-split solution of \eqref{KZ}.

\subsection{Split solutions}

Suppose that $f_k(\tau)\in M_k(SL_2(\Bbb Z))$ satisfies
$$
\bigl(E_4\Bbb D_k^{[2]}+\mu E_6\Bbb D_k+\lambda E_4^2\bigr)f_k=0.
$$
Then $f_k(\tau)\varphi_{0,\frac{3}{2}}(\tau,az)$ and
$f_k(\tau)\varphi_{-1,\frac{1}{2}}(\tau,az)$ are split solutions of \eqref{E4ME}, with weights $k$ and $k-1$, respectively.
\smallskip

\noindent
{\bf Example.} We have
$$
\bigl(E_4H_{6}^{[2]}+\mu E_6H_6+\lambda E_4^2\bigr)
(E_6\varphi_{0,\frac 32})=
(48-6\mu+\lambda)E_4^2E_6\varphi_{0,\frac 32}.
$$
Thus  $E_6\varphi_{0,\frac 32}(az)$ is a spilt solution if $48-6\mu+\lambda=0$.
\smallskip

\begin{proposition}\label{E4Eq-lam0}
When $\lambda=0$, equation \eqref{E4ME} has only split solutions.
\end{proposition}
\begin{proof} We have 
$\bigl(E_4H_{k+2}H_k+\mu E_6H_k\bigr)\varphi_{k,m}=0$. By $\Delta$-homogeneity, we may assume that $q^0[\varphi_{k,m}]\ne 0$.

If $\mu=0$, Corollary \ref{KerH2} and \eqref{KerH} imply that $\varphi_{k,m}$ has split form. The same conclusion holds if $H_k(\varphi_{k,m})=0$. 

Now suppose that $\mu\ne 0$ and $H_k(\varphi_{k,m})\ne 0$. By Theorem \ref{Ind1}, we have $\mu=4b$ and
$$
H_k(\varphi_{k,m})=
C\cdot
\begin{cases}
\varphi_{0,\frac{3}{2}}(\tau,az)E_4^b\Delta^c,\\
\varphi_{-1,\frac{1}{2}}(\tau,az)E_4^b\Delta^c
\end{cases}
$$
for some nonzero constant $C$.

In the proof of \cite[Theorem 2.2]{AG23}, we established that if
$q^0[\varphi_{k,m}]\ne 0$ but $q^0[H_k(\varphi_{k,m})]=0$, then $\varphi_{k,m}$ is a constant multiple of either $\varphi_{0,\frac{3}{2}}(\tau,az)$ or $\varphi_{-1,\frac{1}{2}}(\tau,az)$. Both forms lie in the kernel of the corresponding operator, contradicting our assumption that $H_k(\varphi_{k,m})\ne 0$. Hence $c=0$, and Proposition \ref{Hkphi} implies that $\varphi_{k,m}$ has split form.
\end{proof} 

\subsection{Reduction of the MMDE \eqref{E4ME} to the Kaneko--Zagier equation \eqref{KZ}}

We seek non-split solutions of \eqref{E4ME}. By Proposition \ref{E4Eq-lam0}, we may assume that $\lambda\ne 0$. If $\mu=0$, the equation already reduces to \eqref{KZ}, so we also assume that $\mu\ne 0$.

By $\Delta$-homogeneity, we may suppose that $q^0[\varphi_{k,m}]\ne 0$. Moreover, $H_k(\varphi_{k,m})\ne 0$, since otherwise \eqref{E4ME} and $\lambda\ne 0$ would imply that $\varphi_{k,m}=0$.

The structure of the algebra of weak Jacobi forms and the zeros of $E_4$ and $E_6$ imply, as in \eqref{multE4}, that
\begin{equation}\label{psi-k-2}
H_k(\varphi_{k,m})=E_4\psi_{k-2,m}.
\end{equation}
Substituting this expression into \eqref{E4ME}, applying the product rule, and dividing by $E_4$, we obtain
\begin{equation}\label{phi-psi}
E_4H_{k-2}(\psi_{k-2,m})+(\mu-4)E_6\psi_{k-2,m}
+\lambda E_4\varphi_{k,m}=0.
\end{equation}
If $H_{k-2}(\psi_{k-2,m})=0$, then \eqref{KerH} and Proposition \ref{Hkphi} imply that $\varphi_{k,m}$ has split form.

1) Suppose that $\mu=4$. Dividing \eqref{phi-psi} by $E_4$ and applying $H_k$, we obtain an equation of Kaneko--Zagier type:
\begin{align}\label{Emu4}
\bigl(H_{k-2}^{[2]}+\lambda E_4\bigr)\psi_{k-2,m}=0,
\quad H_k(\varphi_{k,m})=E_4\psi_{k-2,m}
\qquad (\mu=4).
\end{align}
This explains why all the MMDEs considered in \S 2 have $\mu=4$. Moreover, if $\psi_{k-2,m}$ is a split solution, then Proposition \ref{Hkphi} implies that $\varphi_{k,m}$ also has split form.

2) Now suppose that $\mu\ne 4$. We retain the assumption $\lambda\ne 0$, since we seek non-split solutions. It follows from \eqref{phi-psi} that
$$
\psi_{k-2,m}=E_4\xi_{k-6,m},
\qquad H_k(\varphi_{k,m})=E_4^2\xi_{k-6,m}.
$$
Substituting $\psi_{k-2,m}=E_4\xi_{k-6,m}$ into \eqref{phi-psi} and cancelling the factor $E_4$, we obtain
\begin{equation*}\label{PhiXi}
E_4H_{k-6}(\xi_{k-6,m})+(\mu-8)E_6\xi_{k-6,m}
+\lambda\varphi_{k,m}=0.
\end{equation*}
If $H_{k-6}(\xi_{k-6,m})=0$, then $\varphi_{k,m}$ again has split form.

Applying $H_k$ and using $H_k(\varphi_{k,m})=E_4^2\xi_{k-6,m}$ yields a new equation for $\xi_{k-6,m}$:
\begin{equation}\label{e-xik-6}
\bigl(E_4H_{k-6}^{[2]}+(\mu-12)E_6H_{k-6}
+(48-6\mu+\lambda)E_4^2\bigr)\xi_{k-6,m}=0.
\end{equation}
If $48-6\mu+\lambda=0$, Propositions \ref{E4Eq-lam0} and \ref{Hkphi} imply that both $\xi_{k-6,m}$ and $\varphi_{k,m}$ have split form.

3) If $\mu=12$, equation \eqref{e-xik-6} reduces to an equation of Kaneko--Zagier type. If $\mu\ne 12$, we repeat the reduction.

4)If $\mu-12=4$, that is, $\mu=16$, the reduced equation becomes
$$
\bigl(H_{k-8}^{[2]}+(\lambda-48)E_4\bigr)\varepsilon_{k-8,m}=0,
$$
where
$$
H_{k-6}(\xi_{k-6,m})=E_4\varepsilon_{k-8,m},
\qquad
H_k(\varphi_{k,m})=E_4^2\xi_{k-6,m}.
$$

5) If $\mu\ne 16$, the next reduction gives
\begin{equation*}
\bigl(E_4H_{k-12}^{[2]}+(\mu-24)E_6H_{k-12}
+(168-12\mu+\lambda)E_4^2\bigr)\varepsilon_{k-12,m}=0,
\end{equation*}
with
$$
H_{k-6}(\xi_{k-6,m})=E_4^2\varepsilon_{k-12,m},
\qquad
H_k(\varphi_{k,m})=E_4^2\xi_{k-6,m}.
$$
If $\varepsilon_{k-8,m}$ or $\varepsilon_{k-12,m}$ has split form, then Proposition \ref{Hkphi} implies that both $\xi_{k-6,m}$ and $\varphi_{k,m}$ also have split form.

6) Suppose that $\mu$ is neither $12d$ nor $4+12d$ for any $d\in\Bbb Z_{\ge 0}$. Then, at every stage of the reduction, the coefficient of the first-order term $E_6H_*$ is neither $0$ nor $4$, so the reduction can be continued. Since the weights of nonzero weak Jacobi forms of fixed index are bounded below, after sufficiently many steps the resulting form $\varepsilon_{k-12t,m}$ must vanish identically. Tracing the reduction backwards and applying \eqref{KerH} and Proposition \ref{Hkphi}, we conclude that the original solution has split form.
 
We have proved the following result.
\begin{theorem}\label{Th-E4ME}
The equation
\[
\bigl(E_4H_k^{[2]}+\mu E_6H_k+\lambda E_4^2\bigr)\varphi_{k,m}=0
\]
can have a non-split solution only if $\mu=12d$ or $\mu=4+12d$ for some $d\in\Bbb Z_{\ge 0}$. Every such solution is obtained by lifting a non-split solution of a Kaneko--Zagier equation through the reduction process described above.
\end{theorem}

\subsection{Reduction of the $E_6$-MMDE}

We apply the reduction method described above to the second-order MMDE with leading coefficient $E_6$:
\begin{equation}\label{E6ME}
\bigl(E_6H_k^{[2]}+\mu E_4^2H_k+\lambda E_4E_6\bigr)\varphi_{k,m}=0.
\end{equation}
If $\mu=0$, this reduces to a Kaneko--Zagier equation. As in the proof of Proposition \ref{E4Eq-lam0}, we can show that when $\lambda=0$, equation \eqref{E6ME} has only split solutions.

Suppose that $\mu,\lambda\ne 0$. Then $H_k(\varphi_{k,m})=E_6\psi_{k-4,m}$, and substitution gives
$$
E_6H_{k-4}\psi_{k-4,m}+(\mu-6)E_4^2\psi_{k-4,m}+\lambda E_4\varphi_{k,m}=0.
$$
Since we seek non-split solutions, we may assume that $H_{k-4}(\psi_{k-4,m})\ne 0$. Thus, $H_{k-4}\psi_{k-4,m}=E_4\xi_{k-6,m}$ and  the equation  can be rewritten as
$$
E_6\xi_{k-6,m}+(\mu-6)E_4\psi_{k-4,m}+\lambda\varphi_{k,m}=0.
$$
Applying the operator $H_k$ and using the defining relations for $\psi_{k-4,m}$ and $\xi_{k-6,m}$, we obtain
$$
E_6H_{k-6}\xi_{k-6,m}+(\mu-12)E_4^2\xi_{k-6,m}+(\lambda-4\mu+24)E_6\psi_{k-4,m}=0.
$$
Again, we may assume that $H_{k-6}(\xi_{k-6,m})\ne 0$, since otherwise the original solution has split form.

1. Suppose that $\mu=12$. Cancelling the factor $E_6$ in the preceding equation and applying $H_{k-4}$, we reduce \eqref{E6ME} to an equation of Kaneko--Zagier type:
\begin{equation}\label{Ximu12}
\bigl(H_{k-6}^{[2]}+(\lambda-24)E_4\bigr)\xi_{k-6,m}=0,
\end{equation}
where
 where $H_{k-4}\psi_{k-4,m}=E_4\xi_{k-6,m}$ and  
 $H_k\varphi_{k,m}=E_6\psi_{k-4,m}$.

2.Suppose that $\mu\ne 12$. Then $\xi_{k-6,m}$ is divisible by $E_6$, so we write
$\xi_{k-6,m}=E_6\rho_{k-12,m}$.
The equation becomes
$$
\left(E_6H_{k-12}+(\mu-18)E_4^2\right)\rho_{k-12,m}+(\lambda-4\mu+24)\psi_{k-4,m}=0.
$$
For a non-split solution, we again have $H_{k-12}(\rho_{k-12,m})\ne 0$.
Applying $H_{k-4}$ to this equation, we reduce \eqref{E6ME} to
\begin{equation}\label{E6Reccur}
\left(E_6H_{k-12}^{[2]}+(\mu-24)E_4^2H_{k-12}+(\lambda-12\mu+168)E_4E_6\right)\rho_{k-12,m}=0
\end{equation}
where $H_{k-4}\psi_{k-4,m}=E_4E_6\rho_{k-12,m}$   and 
$H_k\varphi_{k,m}=E_6\psi_{k-4,m}$.
As above, we get a result similar to Theorem \ref{Th-E4ME}.

\begin{theorem}\label{Th-E6ME}
The equation  
$\bigl(E_6H_{k}^{[2]}+\mu E_4^2H_k+\lambda E_4E_6\bigr)\varphi_{k,m}=0$
can have a non-split solution only if $\mu=12d$ for some $d\in\Bbb Z_{\ge 0}$. Every such solution is obtained by lifting a non-split solution of a Kaneko--Zagier equation through the reduction process described above.
\end{theorem}

\subsection{Applications to the elliptic genus}

In this subsection, we study possible MMDEs for the elliptic genera of Calabi--Yau varieties.

\begin{theorem}\label{E-CY2d}
Let $M_{2d}$ be a compact complex variety of even dimension $2d$ whose first Chern class vanishes in $H^2(M_{2d},\Bbb R)$. If its arithmetic genus is nonzero, then its elliptic genus $EG(M_{2d};\tau,z)\in J_{0,d}^w$ can satisfy an equation of the form
$$
\bigl(E_4H_0^{[2]}+\mu E_6H_0+\lambda E_4^2\bigr)
EG(M_{2d};\tau,z)=0
$$
only if $2d=2$. In this case, the equation is the first equation in Proposition \ref{Eq-phi01}.
\end{theorem}

{\bf Remarks.} 1. As in \cite{AG23}, we may consider, for example, a (quasi) Calabi--Yau variety $M$ of even dimension $n$, that is, a compact K\"ahler variety, not necessarily simply connected, with trivial canonical class and $h^{p,0}(M)=0$ for $0<p<n$. Its arithmetic genus is then $2$.

2. The elliptic genus of a K3 surface, namely $2\varphi_{0,1}$, satisfies a third-order MDE (see \eqref{eqK3}). In \cite{AG23}, we proved that Calabi--Yau fourfolds whose elliptic genera satisfy a similar third-order MDE fall into only two types, with Euler characteristics $48$ and $-18$. For these fourfolds, passing to MMDEs does not reduce the order of the differential equations. Further details are given in the next section.
\smallskip

We now study second-order  $E_4$-MMDEs  and  $E_6$-MMDEs  for arbitrary
weak Jacobi forms of weight 0. The latter case  is considerably simpler. 
\begin{proposition}\label{E6-k0}
The equation 
$$\bigl(E_6H_{0}^{[2]}+\mu E_4^2H_0+\lambda E_4E_6\bigr)\varphi_{0,m}=0$$
has no non-split solutions in $J_{0,m}^w$ if $\mu\neq 0$ or $\lambda\neq 0$.
\end{proposition}

 We prove Proposition \ref{E6-k0} after Proposition \ref{E4-k0} below. The case of second-order $E_4$-MMDEs is more interesting.

\begin{proposition}\label{E4-k0}
Let $\varphi_{0,m}\in J_{0,m}^w$, with $m\in\frac{1}{2}\Bbb Z$, be a solution of
\begin{align}\label{E4MEk=0}
\bigl(E_4H_0^{[2]}+\mu E_6H_0+\lambda E_4^2\bigr)\varphi_{0,m}=0,
\end{align}
and suppose that $\varphi_{0,m}$ is not of the form $c\varphi_{0,\frac{3}{2}}(\tau,az)$.
Then $\mu=4$, and
$$
q^0[\varphi_{0,m}]=\sum_l a(0,l)\zeta^l\ne 0.
$$
Among the coefficients with non-negative exponents, exactly two are nonzero: $a(0,l_1)$ and $a(0,l_2)$, where $l_1>l_2\ge 0$. Moreover, $m=l_1^2+l_2^2$.
\end{proposition}
\begin{proof}
By Proposition \ref{E4Eq-lam0}, if $\lambda=0$, every solution has split form and is therefore a constant multiple of $\varphi_{0,\frac{3}{2}}(\tau,az)$. If $\mu=0$, the equation is equivalent to
$$
\bigl(H_0^{[2]}+\lambda E_4\bigr)\varphi_{0,m}=0.
$$
In \cite[Theorem 3.3]{AG23}, we proved that this equation has only ``trivial'' split solutions in $J_{0,m}^w$. More precisely, every such solution has the form
$\varphi_{0,m}(\tau,z)=c\varphi_{0,\frac{3}{2}}(\tau,az)$ with $a\in\Bbb N$ and belongs to the kernel of the modular differential operator. Thus, these are the only possible split solutions of the equation in the proposition.
 
 We now assume that $\mu,\lambda\ne 0$. By the reduction process in Theorem \ref{Th-E4ME}, we must consider Kaneko--Zagier equations of negative weights $k=-2,-6,\ldots$. We recall the following result.

\begin{proposition}\label{Pr-KZE}(See \cite[Proposition 3.4]{AG23}.)
Let $\varphi_{k,m}\in J_{k,m}^w$ and
\begin{equation}\label{KZEk}
(H_k^{[2]}+\lambda  E_4)\varphi_{k,m}=0, \quad \lambda\ne 0.
\end{equation}
1. If $a(0,l_1)$, $a(0,l_2)$ are two non-zero Fourier coefficients of 
$\varphi_{k,m}$ with  $l_1^2\ne l_2^2$, then 
$2k+1=-\frac 3m(l_1^2+l_2^2)$.
\smallskip

\noindent
2. There cannot be three nonzero Fourier coefficients
$a(0,l_1)$, $a(0,l_2)$, and $a(0,l_3)$ with pairwise distinct values of $l_i^2$.
\smallskip

\noindent
3. If the nonzero coefficients in the $q^0$-term correspond to exactly two distinct values $l_1^2$ and $l_2^2$, then $k<0$ and
$$
q^0[\varphi_{k,m}]=
\begin{cases}
c(\zeta^{l_1}-\zeta^{l_2}-\zeta^{-l_2}+\zeta^{-l_1})
\quad&\text{for even}\ k,\\
a(\zeta^{l_1}-\zeta^{-l_1})+b(\zeta^{l_2}-\zeta^{-l_2})
\quad&\text{for odd}\ k.
\end{cases}
$$
4. The weight $k\ge -2$ if $k$ is even, and $k\ge -3$ if $k$ is odd.
\end{proposition}

A non-split solution of \eqref{E4MEk=0} can exist only if the corresponding reduced Kaneko--Zagier equation has a non-split solution of weight $k=-2,-6,\ldots$. By Proposition \ref{Pr-KZE}, only $k=-2$ is possible. This weight occurs at the first reduction, when $\mu=4$. Hence $\mu=4$ is the only value for which \eqref{E4MEk=0} can have a non-split solution.

If $q^0[\varphi_{0,m}]=0$, division by $\Delta$, followed by reduction, would yield a nonzero solution of a Kaneko--Zagier equation of weight at most $-12$, contradicting Proposition \ref{Pr-KZE}. Thus $q^0[\varphi_{0,m}]\ne 0$.

For $\mu=4$, equation \eqref{E4MEk=0} reduces, as in \eqref{Emu4}, to
\begin{align}\label{E4MEk=-2}
\bigl(H_{-2}^{[2]}+\lambda E_4\bigr)\psi_{-2,m}=0,
\quad H_0(\varphi_{0,m})=E_4\psi_{-2,m}.
\end{align}
We note that the case $q^0[\psi_{-2,m}]=c (\zeta^l+\zeta^{-l})$
is impossible, since $\psi_{-2,m}$ vanishes to order at least $2$ at $z=0$. Proposition \ref{Pr-KZE} therefore gives
$$
q^0[\psi_{-2,m}]=
c(\zeta^{l_1}-\zeta^{l_2}-\zeta^{-l_2}+\zeta^{-l_1}),\quad
l_1>l_2\ge 0.
$$ 
Moreover, for $k=-2$ we get $m=l_1^2+l_2^2$. 

Write $q^0[\varphi_{0,m}]=\sum_l a(0,l)\zeta^l$. If a nonzero coefficient $a(0,l)$ were vanished by the action of $H_0$ on the $q^0$-term, comparison of the corresponding coefficient in \eqref{E4MEk=0} would force $\lambda=0$, contrary to our assumption. Hence the $q^0$-terms of $\varphi_{0,m}$ and $\psi_{-2,m}$ have the same support. This proves Proposition \ref{E4-k0}.
\end{proof}

The proof of Proposition \ref{E6-k0} is similar. If $\mu=0$, we obtain a Kaneko--Zagier equation considered above. The first possible reduction to a Kaneko--Zagier equation occurs when $\mu=12$ and yields an equation for weight $-6$ (see \eqref{Ximu12}). By Proposition \ref{Pr-KZE}, this equation has no nonzero solutions.
\medskip

We recall (see \cite{Gr99}) the topological interpretation of the Fourier coefficients of the elliptic genus $EG(M;\tau,z)\in J_{0,n/2}^w$ of a compact complex variety $M$ of dimension $n$ with vanishing first Chern class. We have
$$
q^0[EG(M;\tau,z)]=\sum_{p=0}^{n}
(-1)^p\chi(M, \Omega^p_M)
\zeta^{n/2-p}
$$
where $\chi(M, \Omega^p_M)$ denotes the holomorphic Euler characteristic.
 If $M$ is K\"ahler, the coefficient of $\zeta^{n/2}$ is its arithmetic genus,
$\sum_{p=0}^n(-1)^p h^{p,0}(M)$.
For the even-dimensional Calabi--Yau varieties described in the first remark after Theorem \ref{E-CY2d}, this coefficient equals $2$.

By Proposition \ref{E4-k0}, the largest exponent $l_1$ occurring in the $q^0$-term of a non-split solution $\varphi_{0,m}$ of \eqref{E4MEk=0} satisfies $l_1\le\sqrt m$. For the elliptic genus of a variety with nonzero arithmetic genus, we have $m=n/2$ and $l_1=n/2$. Therefore,  $\frac n2=l_1^2+l_2^2$ which, for positive even $n$, is possible only when $n=2$. This proves Theorem \ref{E-CY2d}. 
\smallskip

Proposition \ref{E4-k0} shows that a non-split weight-zero solution of an $E_4$-MMDE has a highly restricted $q^0$-term. Using this restriction, we obtain the following result.

\begin{proposition}\label{E4E-08}
For positive integral or half-integral indices $m\le 8$, equations of the form \eqref{E4MEk=0} have exactly three non-split solutions, up to multiplication by a nonzero constant:
$\varphi_{0,1}(\tau,z)$, $\varphi_{0,\frac{5}{2}}(\tau,z)$, and $\varphi_{0,1}(\tau,2z)$, of indices $1$, $\frac{5}{2}$, and $4$, respectively.
\end{proposition}
\begin{proof} Let $\psi_{-2,m}$ be a corresponding solution of 
\eqref{E4MEk=-2}. We have
$\psi_{-2,m}=\varphi_{-2,1}\psi_{0,m-1}$, where
$q^0[\varphi_{-2,1}]=\zeta^{\pm 1}-2$.
For $m<6$ and $m=\frac{13}{2}$, the form $\psi_{0,m-1}$ is uniquely determined by its $q^0$-term (see \cite{Gr99}).

For $m\le 8$, the relation $m=l_1^2+l_2^2$, with admissible integral or half-integral exponents $l_1>l_2\ge 0$, is possible only for
$m=1,\frac{5}{2},4,5,\frac{13}{2}$.
The first two values give the solutions $\varphi_{0,1}$ and $\varphi_{0,\frac{5}{2}}$, whose equations are given in Proposition \ref{Eq-phi01}. These forms are proportional to the elliptic genera of K3 surfaces and Calabi--Yau fivefolds, respectively. For $m=4$, we obtain the rescaled form $\varphi_{0,1}(\tau,2z)$ of index $4$.

For $m=5$ and $m=\frac{13}{2}$, the prescribed $q^0$-terms determine the corresponding Jacobi forms uniquely up to multiplication by a constant. These forms are
$$
\psi_{-2,5}=\varphi_{-2,1}\varphi_{0,4}=\zeta^{\pm 2}-\zeta^{\pm 1}+O(q)
$$
where 
$
\varphi_{0,4}(\tau, z)=\zeta+1+\zeta^{-1}+O(q)=
{\vartheta(\tau, 3z)}/{\vartheta(\tau, z)},
$
and 
$$
\psi_{-2,\frac {13}{2}}=\psi_{-2,5}\varphi_{0,\frac 32}=
\zeta^{\pm \frac 52}-\zeta^{\pm \frac 12}+O(q).
$$
A direct comparison of the $q^1$-terms shows that neither $\psi_{-2,5}$ nor $\psi_{-2,\frac{13}{2}}$ satisfies a Kaneko--Zagier equation of weight $-2$.
\end{proof}

\section{MMDEs for weak Jacobi forms of weight 0 and index 2}

In \cite{AG23}, we studied monic MDEs of minimal order for weak Jacobi forms in $J_{0,2}^w$. In this section, we determine their MMDEs of minimal order.

Recall that every $\Phi_{0,2}\in J_{0,2}^w$ can be written as
$$
\Phi_{0,2}=a\psi_{0,2}+b\varphi_{0,2}=a\zeta^{\pm 2}+b \zeta^{\pm 1}+2(11 a+2b)+$$
$$+q \left(b\zeta^{\pm 3}+8(29 a-b) \zeta^{\pm 2}-(1024 a+b) \zeta^{\pm 1}+16(99 a+ b)\right)+O(q^2)
$$
where the two basis elements are
$$\psi_{0,2}=\zeta^{\pm 2}+22+q\left(232\zeta^{\pm 2}-1024\zeta^{\pm 1}+1584\right)+O(q^2),$$
$$\varphi_{0,2}=\zeta^{\pm 1}+4+q\left(\zeta^{\pm 3}-8\zeta^{\pm 2}-\zeta^{\pm 1}+16\right)+O(q^2).$$
We prove the following result.

\begin{theorem}
Any nonzero weak Jacobi form $\phi(\tau,z)\in J_{0,2}^w$ satisfies an explicitly determined third-order MMDE. Up to multiplication by a nonzero constant, the possibilities are as follows.
\smallskip

\noindent
1) Exactly three forms, $\varphi_{0,2}$, $\psi_{0,2}$, and $\rho_{0,2}=2\psi_{0,2}-11\varphi_{0,2}$, satisfy a monic third-order MDE.
\smallskip

\noindent
2)  Exactly one form,
$\omega_{0,2}=5\psi_{0,2}-308\varphi_{0,2}$,
satisfies the $E_4$-MMDE \eqref{omega}.
\smallskip

\noindent
3)  The two forms
$\alpha_{0,2}^{\pm}=\psi_{0,2}+(65\pm\frac{3}{5}\sqrt{13105})\varphi_{0,2}$
satisfy the $E_6$-MMDE \eqref{E6MMDEa}.
\smallskip

\noindent
4) Exactly one form,
$\beta_{0,2}=\psi_{0,2}+44\varphi_{0,2}$,
satisfies the $E_8$-MMDE \eqref{E8MMDE}.
\smallskip

\noindent
5) Every other form $\Phi_{0,2}\in J_{0,2}^w$ satisfies the MMDE \eqref{wt12MMDE}, whose leading coefficient $E_4^3+\lambda_4\Delta$ has weight $12$.
In particular, $\lambda_4=0$ for $$\Phi_{0,2}=\psi_{0,2}-4\varphi_{0,2}=
E_4\varphi_{-2,1}^2=E_4(\tau)\vartheta(\tau,z)^4/\eta(\tau)^{12}.$$
\end{theorem}

The results of the preceding sections show that no nonzero Jacobi form in $J_{0,2}^w$ satisfies an $E_4$- or $E_6$-MMDE of order one or two. In \cite[Theorem 4.1]{AG23}, we proved that, up to multiplication by a nonzero constant, exactly three forms --- $\varphi_{0,2}$, $\psi_{0,2}$, and $2\psi_{0,2}-11\varphi_{0,2}$ --- satisfy a monic third-order MDE.

To analyze MMDEs, we need explicit formulae for the action of the modular differential operator. By \eqref{Hk-act}, the $q^0$- and $q^1$-terms of the Fourier expansions are
\begin{align*}
q^{0,1}[H_0(\Phi_{0,2})]
&=-\tfrac{11}{2}a\zeta^{\pm 2}-b\zeta^{\pm 1}+11a+2b\\
&\quad+q\bigl(-b\zeta^{\pm 3}+(1496a-52b)\zeta^{\pm 2}\\
&\qquad\qquad-(11264a-23b)\zeta^{\pm 1}+19536a+152b\bigr),
\end{align*}
\begin{align*}
q^{0,1}[H_0^{[2]}(\Phi_{0,2})]
&=\tfrac{165a}{4}\zeta^{\pm 2}+3b\zeta^{\pm 1}
-\tfrac{3}{2}(11a+2b)\\
&\quad+q\bigl(3b\zeta^{\pm 3}+(6534a-234b)\zeta^{\pm 2}\\
&\qquad\qquad-(101376a-243b)\zeta^{\pm 1}+205524a+1668b\bigr),
\end{align*}
and
\begin{align*}
q^{0,1}[H_0^{[3]}(\Phi_{0,2})]
&=-\tfrac{3135}{8}a\zeta^{\pm 2}-15b\zeta^{\pm 1}
+\tfrac{21}{4}(11a+2b)\\
&\quad+q\bigl(-15b\zeta^{\pm 3}+(19800a-585b)\zeta^{\pm 2}\\
&\qquad\qquad-(709632a+1449b)\zeta^{\pm 1}+1745568a+13926b\bigr).
\end{align*}
Having treated the three exceptional forms above, we may assume that $a$, $b$, and $11a+2b$ are all nonzero. We now systematically examine the possible leading coefficients of MMDEs, beginning with the Eisenstein series $E_4$.
\smallskip

\textbf{1. The $E_4$ case.} Suppose that $\Phi_{0,2}$ satisfies
\begin{equation}\label{E4MMDE}
\bigl(E_4H_0^{[3]}+\lambda_1 E_6 H_0^{[2]}+\lambda_2E_4^2H_0+\lambda_3E_4E_6\bigr)\Phi_{0,2}=0.
\end{equation}
Comparing the $q^0$-terms gives the linear system
$$
\left\{
\begin{aligned}
\tfrac{165}{4}\lambda_1-\tfrac{11}{2}\lambda_2+\lambda_3& =\tfrac{3135}{8}\\
3\lambda_1-\lambda_2+\lambda_3& =15\\
\tfrac{3}{2}\lambda_1-\lambda_2-2\lambda_3& =\tfrac{21}{4}.
\end{aligned}
\right.
$$
Its unique solution is $\lambda_1=12$, $\lambda_2=\frac{73}{4}$, and $\lambda_3=-\frac{11}{4}$.
Substituting these values into the operator in \eqref{E4MMDE} and calculating the $q^1$-term, we obtain
\begin{align*}
&\bigl(E_4H_0^{[3]}+12E_6H_0^{[2]}+\tfrac{73}{4}E_4^2H_0-\tfrac{11}{4}E_4E_6\bigr)\Phi_{0,2}\\
&\qquad=-864(308a+5b)q(\zeta^{\pm 2}+8\zeta^{\pm 1}-18)+O(q^2).
\end{align*}
Since a weak Jacobi form of weight $10$ and index $2$ is uniquely determined by its $q^0$- and $q^1$-terms, this gives the identity
$$
\bigl(E_4H_0^{[3]}+12E_6H_0^{[2]}+\tfrac{73}{4}E_4^2H_0-\tfrac{11}{4}E_4E_6\bigr)\Phi_{0,2}
=-864(308a+5b)\Delta\varphi_{0,1}\varphi_{-2,1}.
$$
Thus, up to multiplication by a nonzero constant, the only form satisfying a third-order $E_4$-MMDE is
$\omega_{0,2}=5\psi_{0,2}-308\varphi_{0,2}$.
Dividing by $E_4$, we can write its equation as
\begin{equation}\label{omega}
\bigl(H_0^{[3]}+12 \frac{E_6}{E_4} H_0^{[2]}+ \frac{73}{4}E_4H_0-\frac{11}{4}E_6\bigr)\omega_{0,2}=0.
\end{equation}

Note that the form $\omega_{0,2}$ is also the most exceptional case among the forms in $J_{0,2}^w$ with respect to monic MDEs: it satisfies the sixth-order equation
\begin{multline*}\label{HH6}
\bigl(H^{[6]}_0
-\tfrac{983}{4}E_4H^{[4]}_0
+\tfrac{357}{2}E_6H^{[3]}_0+\\
\tfrac{11343}{2}E_4^2H^{[2]}_0
+\tfrac{2069}{16}E_4E_6H_0+
\tfrac{19217}{16}E_6^2\bigr)\,\omega_{0,2}=0.
\end{multline*}
For this form, passing from a monic MDE to a non-monic one gives the greatest reduction in order.
\medskip

\textbf{2. The $E_6$ case.}
Suppose that $\Phi_{0,2}$ satisfies
\begin{equation*}\label{E6MMDE}
\bigl(E_6H_0^{[3]}+\lambda_1E_4^2H_0^{[2]}+\lambda_2E_4E_6H_0
+(\lambda_3E_4^3+\lambda_4\Delta)\bigr)\Phi_{0,2}=0.
\end{equation*}
The Eisenstein series have constant term 1, whereas $\Delta$ has constant term $0$. Comparing the $q^0$-terms therefore gives the same linear system as in the $E_4$ case. Substituting the values of $\lambda_1$, $\lambda_2$, and $\lambda_3$ found above, we obtain
\begin{align*}
&\bigl(E_6H_0^{[3]}+12E_4^2H_0^{[2]}+\tfrac{73}{4}E_4E_6H_0
-\tfrac{11}{4}E_4^3+\lambda_4\Delta\bigr)\Phi_{0,2}\\
&\quad=q\Bigl((a\lambda_4+584496a-4320b)\zeta^{\pm 2}\\
&\qquad\qquad+(b\lambda_4-2128896a+22896b)\zeta^{\pm 1}\\
&\qquad\qquad+2(11a+2b)\lambda_4+4343328a-3456b\Bigr)+O(q^2).
\end{align*}
Since a weak Jacobi form of weight 12 and index 2 is uniquely determined by its $q^0$- and $q^1$-terms, this expression vanishes identically if and only if
$$
\left\{
\begin{aligned}
a\lambda_4+ 584496a-4320b&=0\\
b\lambda_4-2128896a+22896b& =0\\
2(11a+2b)\lambda_4+4343328a-3456b &=0.
\end{aligned}
\right.
$$
Solving this system with a computer algebra system gives 
\begin{equation}\label{sqrtab2}
b=\bigl(65\pm \tfrac{3}{5}\sqrt{13105}\bigr)a, \quad 
\lambda^{\pm}_4 =432 \bigl(-703\pm6 \sqrt{13105}\bigr)
\end{equation} 
where signs $\pm$ coincide. Thus, up to multiplication by a nonzero constant, there are exactly two exceptional Jacobi forms,  
$$\alpha^{\pm}_{0,2}=\psi_{0,2}+(65\pm \tfrac{3}{5}\sqrt{13105})\varphi_{0,2},$$
satisfying a third-order $E_6$-MMDEs
\begin{equation}\label{E6MMDEa}
\bigl(E_6H_0^{[3]}+12 E_4^2 H_0^{[2]}+\tfrac{73}{4}E_4E_6H_0+\bigl(-\tfrac{11}{4}E_4^3+\lambda^{\pm}_4\Delta\bigr)\bigr)\alpha^{\pm}_{0,2}
=0.
\end{equation}

\textbf{3. The $E_4^2$ case.}
Suppose that $\Phi_{0,2}$ satisfies
\begin{equation*}\label{E8MMDE}
\bigl(E_4^2H_0^{[3]}+\lambda_1 E_4E_6 H_0^{[2]}+(\lambda_2E_4^3+\lambda_4\Delta)H_0+\lambda_3E_4^2E_6\bigr)\Phi_{0,2}=0.
\end{equation*}
As in the $E_4$ and $E_6$ cases, comparison of the $q^0$-terms gives
$\lambda_1=12$, 
$\lambda_2=\frac{73}{4}$, 
$\lambda_3=-\frac{11}{4}$. 
Substituting these values and calculating the $q^1$-term, we obtain
\begin{align*}
&\bigl(E_4^2H_0^{[3]}+12E_4E_6H_0^{[2]}
+(\tfrac{73}{4}E_4^3+\lambda_4\Delta)H_0
-\tfrac{11}{4}E_4^2E_6\bigr)\Phi_{0,2}\\
&\quad=q\Bigl(-\bigl(\tfrac{11}{2}a\lambda_4+266112a+4320b\bigr)\zeta^{\pm 2}\\
&\qquad\qquad-(b\lambda_4+2128896a+34560b)\zeta^{\pm 1}\\
&\qquad\qquad+(11a+2b)\lambda_4+4790016a+77760b\Bigr)+O(q^2).
\end{align*}
 A weak Jacobi form of weight $14$ and index $2$ is uniquely determined by its $q^0$- and $q^1$-terms. Solving the resulting system with a computer algebra system shows that the expression vanishes identically only in the following two cases: $b=44a$, $\lambda_4=-82944$ and $5b=-308a$, $\lambda_4=0$. The first case gives the exceptional form $\beta_{0,2}=\psi_{0,2}+44\varphi_{0,2}$, which satisfies the $E_4^2$-MMDE
\begin{equation}\label{E8MMDE}
\bigl(E_4^2H_0^{[3]}+12E_4E_6H_0^{[2]}
+(\tfrac{73}{4}E_4^3-82944\Delta)H_0
-\tfrac{11}{4}E_4^2E_6\bigr)\beta_{0,2}=0.
\end{equation}
The second case, $5b=-308a$, follows directly from the $E_4$-MMDE \eqref{omega}.
 
\textbf{4. The $E_4E_6$ case.} We next consider
\begin{multline*}
(E_4E_6H_0^{[3]}+(12E_4^3+\lambda_4\Delta) H_0^{[2]}
+\tfrac{73}4E_4^2E_6H_0 \\
+(-\tfrac{11}4E_4^4+\lambda_5E_4\Delta))\Phi_{0,2}=0
\end{multline*}
where we have used the coefficients determined by our previous comparison of the $q^0$-terms. Calculating the $q^1$-term gives
\begin{align*}
&\bigl(E_4E_6H_0^{[3]}+(12E_4^3+\lambda_4\Delta)H_0^{[2]}
+\tfrac{73}{4}E_4^2E_6H_0
-\tfrac{11}{4}E_4^4+\lambda_5E_4\Delta\bigr)\Phi_{0,2}\\
&\quad=q\Bigl((\tfrac{165}{4}a\lambda_4+a\lambda_5+584496a-4320b)\zeta^{\pm 2}\\
&\qquad\qquad+(3b\lambda_4+b\lambda_5-2128896a+22896b)\zeta^{\pm 1}\\
&\qquad\qquad-\tfrac{3}{2}(11a+2b)\lambda_4
+2(11a+2b)\lambda_5+4343328a-3456b\Bigr)+O(q^2).
\end{align*}
Again, a weak Jacobi form of weight $16$ and index $2$ is uniquely determined by its $q^0$- and $q^1$-terms. The corresponding system has only the solutions $5b=-308a$, $\lambda_4=-20736$, $\lambda_5=4752$, and 
$b=(65\pm \tfrac{3}{5}\sqrt{13105})a$, $\lambda_4=0$, 
$\lambda_5=432(\pm 6 \sqrt{13105}-703)$ with matching signs. All these solutions correspond to forms for which we have already obtained MMDEs, so this case yields no new exceptional forms.
\smallskip

\textbf{5. The case of weight 12.} 
Finally, we consider MMDEs with leading coefficient of weight $12$:
\begin{multline}\label{wt12MMDE}
\left((E_4^3+\lambda_4\Delta)H_0^{[3]}+12E_4^2E_6 H_0^{[2]}
+(\tfrac{73}4E_4^4+\lambda_5E_4\Delta)H_0+\right.\\
\left.+(-\tfrac{11}4E_4^3E_6+\lambda_6E_6\Delta)\right)\Phi_{0,2}=0.
\end{multline}
The $q^1$-term of the left-hand side is
\begin{align*}
q\Bigl(
&-\bigl(\tfrac{3135}{8}a\lambda_4+\tfrac{11}{2}a\lambda_5-a\lambda_6
+266112a+4320b\bigr)\zeta^{\pm 2}\\
&-(15b\lambda_4+b\lambda_5-b\lambda_6+2128896a+34560b)\zeta^{\pm 1}\\
&+\tfrac{21}{4}(11a+2b)\lambda_4+(11a+2b)\lambda_5
+2(11a+2b)\lambda_6+4790016a+77760b
\Bigr).
\end{align*}
A weak Jacobi form of weight $18$ and index $2$ is uniquely determined by its $q^0$- and $q^1$-terms, since there are no nonzero weak Jacobi forms of weight $-6$ and index $2$. The resulting linear system in $\lambda_4$, $\lambda_5$, and $\lambda_6$ has the unique solution
\begin{align*}
\lambda_4&=\frac{16(44a-b)(4a+b)(308a+5b)}{3ab(11a+2b)},\\
\lambda_5&=\frac{4(308a+5b)(46288a^2+12680ab-47b^2)}{3ab(11a+2b)},\\
\lambda_6&=\frac{4(44a-b)(308a+5b)(484a+13b)}{3ab(11a+2b)}.
\end{align*}
The excluded cases $a=0$, $b=0$, and $11a+2b=0$ correspond to the forms satisfying monic third-order MDEs. The constants $\lambda_i$ can also be expressed in terms of the ratio $c=b/a$. In particular, $\lambda_4=0$ when $c=-4$.

\medskip

\noindent
{\bf The simplest MMDEs for the elliptic genera of Calabi--Yau fourfolds.}

We conclude the paper by presenting the simplest differential equations for the elliptic genera of strict Calabi--Yau fourfolds. By \cite{Gr99}, we have
$EG(M_4;\tau,z)=2\psi_{0,2}+b\varphi_{0,2}$.
Thus, the elliptic genus is uniquely determined by the Euler characteristic $e(M_4)=48+6b$.

Theorem 5.1 gives three special elliptic genera satisfying particularly simple differential equations:
$$
2\psi_{0,2},\quad \rho_{0,2},\quad 2\beta_{0,2},
\qquad\text{for }e(M_4)=48,\,-18,\,576,
$$
respectively. For the first two forms, monic third-order MDEs were constructed in \cite[(4.4) and (4.5)]{AG23}:
\begin{equation*}\label{HkEpsi02}
\bigl(H_0^{[3]}-\tfrac{263}{4}E_4H_0+\tfrac{121}{4}E_6\bigr)\psi_{0,2}=0,
\end{equation*}
\begin{equation*}\label{HErho02}
\bigl(H_0^{[3]}-\tfrac{335}{4}E_4H_0-\tfrac{275}{4}E_6\bigr)\rho_{0,2}=0.
\end{equation*}
The third exceptional equation is the $E_4^2$-MMDE \eqref{E8MMDE} for $\beta_{0,2}$, which can be rewritten as
\[
\bigl(H_0^{[3]}+12\frac{E_6}{E_4}H_0^{[2]}
+\bigl(48\frac{E_6^2}{E_4^2}-\frac{119}{4}E_4\bigr)H_0
-\frac{11}{4}E_6\bigr)\beta_{0,2}=0.
\]
The Jacobi form $2\beta_{0,2}$ is the elliptic genus of a strict Calabi--Yau fourfold with Euler characteristic $576=24^2$. Such a variety was constructed in \cite{EJK} as an elliptically fibered Calabi--Yau fourfold. In $F$-theory, it is dual to a gauge theory with gauge group $SO(8)^2$.

A further example is the form $E_4\varphi_{-2,1}^2$, which satisfies an $E_4^3$-MMDE. Dividing this equation by $E_4^3$ introduces the additional meromorphic term $E_6^3/E_4^3$. Since $\varphi_{-2,1}(\tau,0)=0$, a strict Calabi--Yau fourfold whose elliptic genus is proportional to this form would have Euler characteristic $0$.
\medskip

\noindent
{\bf Acknowledgements.}
The study has been funded within the framework of the HSE University Basic Research Program (HSE-BR-2025-060).

\bibliographystyle{amsplain}

\end{document}